\documentclass[12pt,a4paper]{amsart}
\usepackage{preamble}

\title{Quotients of Nash Manifolds by Equivalence Relations}

\begin{document}
\date{\today}
\author[A. Aizenbud]{Avraham Aizenbud}
\address{Avraham Aizenbud,
Faculty of Mathematics and Computer Science, Weizmann
Institute of Science, POB 26, Rehovot 76100, Israel }
\email{aizenr@gmail.com}
\urladdr{http://www.wisdom.weizmann.ac.il/~aizenr}

\author[S. Carmeli]{Shachar Carmeli}

\address{Shachar Carmeli, Department of Mathematical Sciences, University of Copenhagen, Universitetsparken 5
DK-2100 Copenhagen Ø, Denmark}

\author[D. Gourevitch]{Dmitry Gourevitch}
\address{Dmitry Gourevitch, Faculty of Mathematics and Computer Science, Weizmann
Institute of Science, POB 26, Rehovot 76100, Israel }
\email{dimagur@weizmann.ac.il}
\urladdr{http://www.wisdom.weizmann.ac.il/~dimagur}

\author[L. Radzivilovski]{Lev Radzivilovski}
\address{Lev Radzivilovski, Department of Mathematics, Tel Aviv University, P.O. Box 39040, Tel Aviv 6997801, Israel}

\maketitle

\begin{abstract}
We prove that a quotient of a Nash manifold $X$ by a closed equivalence relation $R\subset X\times X$, which is submersive over $X$, yields a Nash manifold $X/R$.
%We then revisit the notions of Nash spaces and Nash stacks from \cite{sakellaridis},
%show that every separated Nash space is representable by a Nash manifold, and that a Nash stack with separated diagonal can be represented by a Nash groupoid.
\end{abstract}

\tableofcontents

%\subfile{sections/schwartzOnNash}

\section{Introduction}
\label{sec:introduction}
\subsection{Main Results}
%\rcomment{bla bla}
The main result of this paper is the following theorem:
\begin{introthm}[See \S\ref{sssec:reg.rel} below]\label{thm:main}
Let $X$ be a Nash manifold\footnote{i.e. infinitely smooth semi-algebraic manifolds. See e.g. \cite{Shiota,aizenbud2008schwartz} for more details.} and let $R\subseteq X\times X$ be an equivalence relation on $X$. Assume that $R$ is a closed, Nash submanifold of $X\times X$, and that the projection $R\to X$ is a Nash submersion. Then, the quotient  $X/R$ exists in the category of Nash manifolds, the quotient morphism $X\onto X/R$ is a Nash submersion, and the canonical morphism 
$R\to X\times_{X/R} X$ is a Nash diffeomorphism.   
\end{introthm}

Traditionally, Nash manifolds are required to be Hausdorff, so the requirement that $R$ is closed is necessary. However one can redefine the notion of Nash manifolds without the Hausdorfness assumption and in this case the requirement that $R$ is closed is not necessary in general. We do not know whether  \Cref{thm:main} holds when dropping the closedness assumption in this case.

In order to prove \Cref{thm:main} we use the notion of complexity for semi-algebraic subsets in $\R^n$ and prove basic properties of it. Roughly speaking, we show that any algebraic procedure cannot increase the complexity in an uncontrollable way (see \Cref{ssec:Tame_Functions}), and that the degree of the  Zariski closure of a set with a bounded complexity is itself bounded (see \Cref{degree_pol_vanishing_controlled}).

\subsection{Background and Motivation}
One basic construction in geometry is the quotient of a space by an equivalence relation. Many examples of geometric objects, such as the projective space, the Grassmanian, moduli spaces, e.t.c. are described as such quotients. In general, one can describe algebraic varieties as quotients of affine algebraic varieties by certain type of equivalence relations.

However, even a quotient of a nice geometric object by an equivalence relation can be very wild. Hence, to guarantee that the quotient is well-behaved, one has to impose conditions on the equivalence relation.
One basic example is the formation of quotients in the category of Hausdorff topological spaces. A quotient of a  Hausdorff topological space by an equivalence relation always exists as a topological space, but it is Hausdorff if and only if the equivalence relation is \emph{closed}.

In the category of $C^\infty$-manifolds, one has the following criterion for the existence of quotients: 

\begin{theorem}[{{\cite[\S 5.9.5]{bourbaki2007varietes}}}]\label{thm:deff}
A quotient of a  $C^\infty$-manifold $X$ by  an equivalence relation $R$ is a  $C^\infty$-manifold if the projection  $R\to X$ is a submersion \footnote{usually one requires the Hausdorff property from $C^\infty$-manifolds. In this case one has also to require that $R$ is closed}.
\end{theorem}

The immediate analog of the above criterion does not hold for algebraic varieties (or schemes).  This fact is the main motivation for introducing algebraic spaces. Namely, one can define an algebraic space as the quotient of a scheme $X$ by an equivalence relation $R$ for which the projection $R\to X$ is a smooth morphism (i.e., a submersion). Then, the natural analog of \Cref{thm:deff}  does hold for algebraic spaces. 

The main difference between the Nash setting of \Cref{thm:main} and the $C^\infty$-setting of \Cref{thm:deff} is that a Nash manifold is required to have a \emph{finite} cover by affine open submanifolds. Indeed, as in \Cref{thm:deff}, it is easy to construct an infinite open cover of the quotient $X/R$ by Nash manifolds, and the main difficulty is to present a construction which gives a finite cover.  

%expects that the smei-algebraic setting is somehow closer the deferential topolgy rather then to the algebraic geomtry setting, since the implisit function thoerem holds in smei-algebraic setting \RamiQ{??? do we have an example that shows tha tit is not aoutomatic??why}. Theorem \ref{thm:main} serves as an analoge of Therem \ref{thm:deff}  to the setting of Nash manifods, under the additinal asumption of Housdorfness.

\subsection{Idea of the proof of Theorem \ref{thm:main}}
In the case when $R\to X$ is \`{e}tale, Theorem \ref{thm:main} can be deduced from the considerations in \cite{bakker2018minimal}. 
 It is therefore enough to show that one can cover $X/R$ by finitely many open sets, each of which is a quotient of a Nash manifold $Y_i$ by an equivalence relation $R_i$ whose projection to $Y_i$ is \'{e}tale 

We do this by introducing the notion of \emph{\'{e}tale slice} (See Definition \ref{def:slice} below). An \'{e}tale slice to $R$ is a submanifold $S\subset X$ such that the restriction $R|_S$ of $R$ to $S$ is \'{e}tale over $S$ and the quotient $S/{R|_S}$ is open in $X/R$. 
Now, it is enough to find finitely many (Nash)  \'{e}tale slices to $R$ whose union covers $X$ modulo $R$ (see \Cref{prop:slis.reg}). Standard technique of differential topology allows to find such an \emph{infinite} collection of slices. It is easy to ensure that those slices will be Nash, but it is more tricky to ensure that their number is finite.

We may assume that $X=\RR^n$. Our construction of the slices is the following: For any $p\in X$ we construct a slice that we denote by $V(R,p)$. Roughly speaking, $V(R,p)$ consists of foots of perpendiculars from $p$ to equivalence classes of $R$, satisfying a suitable transversality condition (see Definition \ref{defn:VRp}).
%\RamiQ{is the following correct "to be the colection of all points $x$ s.t. $p$ is not a focal (see definition \ref{def:foc}) point to the equvalence class of $y$."} 
It is easy to see that, when we range over all $p\in X$, the slices $V(R,p)$ form a collection of slices as required, except of being an infinite collection.

We next show that if we choose a large enough amount of points $p$ in general position, we already cover all of $X$ (modulo $R$). For this, we introduce the notion of complexity of a semi-algebraic set (See \S\ref{ssec:comp} below). We show that the set of points $p$ for which an equivalence class $R_x$ does not meet $V(R,p)$ is of bounded complexity (see \Cref{sec:focal.points}). We then show that for every integer $C$, there exist a finite set $\Sigma$ of points such that every subset of complexity at most $C$ misses at least one point of $\Sigma$ (see \Cref{avoid_small_complexity}).
%\begin{rmk}
%Infact one can make a similar argument for the case $R\to X$ is etal (and by this to replace the usege of \cite{??}). In the etal case one should replace the set $S_{x,R}$ by the set of all points wich are (strictly) closer to $x$ in thiere equvalence class. This will be a open set in wich the relation $R$ is trivial (rathere then a slice to $R$). The rest of the argument is similar and simpler \RamiQ{(?? I do not see the pigen hall argument any more, its look like it more natural to see it with the complexity of hiperplane arengment??)}
%\end{rmk}
%\subsection{limitation of our argument}
%Then main weekness of our argument that it is not valid in the non-Housdorff case.
%??Is it steel the case? may be the new argument does work there. Do we have an example when $S_{x,R}$ is not a slice?
%The resone is that we can not p
\subsection{Possible  Extensions and Questions}
Theorem \Cref{thm:closed_submersive_then_regular}  requires that the equivalence relation $R$ is closed. This naturally raises the following question: 

\begin{question}
Let $X$ be a Nash manifold and $R$ a submersive, not necessarily closed equivalence relation on $X$. Does the quotient $X/R$ exist in the category of (non-Hausdorff) Nash manifolds?
\end{question}

Semi-algebraic sets form an $o$-minimal structure.
\begin{question}
Let $X$ be a smooth manifold definable in a given $o$-minimal structure, and let $R$ be a definable submersive closed equivalence relation on $X$. Does the quotient $X/R$ exist in the category of definable smooth manifolds for this structure?
\end{question}
The answer for this question is positive when $R$ is an \'{e}tale equivalence relation, as shown for complex analytic spaces in \cite[Corollary 2.58]{bakker2018minimal}, and the proof works also in the $C^\infty$-settings. 

\subsection{Motivation and Possible Applications}
This paper is motivated by an attempt to better understand the notion of Nash stacks. The notion of algebraic stacks can be defined using groupoid objects in algebraic spaces. In turn, algebraic spaces can be defined as quotients of schemes by submersive (a.k.a. smooth) equivalence relations.  
On the other hand, in differential topology, differentiable stacks are usually defined directly using groupoid objects in smooth manifolds. This discrepancy is justified by the fact that smooth manifolds admit quotients by closed\footnote{in this context, one usually excludes non-Hausdorff manifolds and the equivalence relation is forced to be closed.} submersive equivalence relations.  

In the Nash context, the differential topology convention is taken in \cite{sakellaridis} for the definition of Nash stacks. The present paper similarly serves as a justification for this choice. We thus expect that the results of this paper can be used to show that a more general definition of Nash stacks, in the flavor of algebraic geometry, is equivalent to the one in \cite{sakellaridis} (at least under suitable separability assumptions). 

%\cite{sakellaridis} it is defined using Nash groupoids, in algebraic geometry 
%Theorem \ref{thm:main} means that unlike the algebric seting, in the Nash setting the analoge of the notion of algebraic space does not add any new geometric objects, under the sumption of Hosdrfness, as stated in Corollary \ref{cor:Nash.sp}.
%\scomment{need to rephrase from here}
%Theorem \ref{thm:main} also influance the notion of Nash stack, whael in the algebraic setting one have to first define algebraic space, and then Define the notion of algebraic stack based on it (see e.g. \cite{LM,Stak}\footnote{there are modern aproches which do this 2 steps at once, but they are equvivalent}), in the Nash setting (at list under Housdorffness asumption) one can take a shortcut and define Nash stack directly based on Nash manifolds, as stated in Corolary \ref{cor:Nash.stak}. This shortcat was taken in \cite{Sak}, so Corolary \ref{cor:Nash.stak} serves as a (partial) justification for it.

%Spesial cases of Corolary \ref{cor:real} was proven in \cite{Sak,AA_stacks}. As expleand in \S\S\ref{ssec:real}, our results together  \cite{Sak,AA_stacks} 
 %implies %Corolary \ref{cor:real}

%\begin{itemize}
%    \item General definable manifolds in $O$-mimimal structures?
%    \item Non-closed equivalence relations 
%    \item 
%\end{itemize}

\subsection{Acknowledgements}
We thank Yiannis Sakellaridis for many interesting discussions on Nash stacks which led us to think about this problem.  
We thank Gal Binyamini for referring us to the relevant sources in semi-algebraic complexity theory. 
A.A. and D.G. were partially supported by ISF grant 249/17.
A.A. was also partially supported by a Minerva foundation grant.
S.C. was partially supported by the Danish National Research Foundation
through the Copenhagen Centre for Geometry and Topology (DNRF151).

%\subsection{Abstract Theory}
%\begin{defn}
%Let $\cC$ be a category which admits finite limits. An \tdef{equivalence relation} in $\cC$ is a groupoid object $R \rightrightarrows X$ such that the map $(s,t)\colon R \to X\times X$ is a monomorphism. We refer to $R$, in this case, as an equivalence relation on $X$. The equivalence relation $R\rightrightarrows X$ is called \tdef{effective} if its co-equalizer (or ''quotient") $X/R$ exists, and the canonical map $R \to X\times_{X/R} X$ is an isomorphism in $\cC$.  
%\end{defn}

%\begin{example}
%Every equivalence relation on a set is effective. More generally, every equivalence relation in a sheaf category $\Shv(X,\Sets)$ is effective. 
%\end{example}

\section{Preliminaries from Semi-algebraic Geometry}
In this section we review and expand the basics of semi-algebraic geometry that we shall need. Mainly, we introduce the notion of \emph{complexity} of a semi-algebraic set, closely related to that of description complexity. We then study focal points of embedded Nash manifolds and their complexity.  

\subsection{Complexity and Tameness}\label{ssec:comp}
We shall now study two related notions, that of complexity of a semialgebraic set (or morphism), and the notion of tame function between collections of semi-algebraic sets. The philosophy is that complexity can not grow unwieldy under operations on semi-algebraic sets that can be described in terms of finitely many polynomial procedures.  
\subsubsection{Complexity}
\begin{defn}
Let $X$ be a semialgebraic set and let $N>0$.
\begin{itemize}
    \item We say that $X$ is \tdef{$N$-basic} if it can be written as 
\[   
    \{x\in \RR^n : g_1(x)=0,\dots,g_m(x)=0,g_{m+1}(x)>0,\dots,g_N(x)>0\}
\]
such that the $g_i$-s are polynomials in $n\le N$ variables of degree at most $N$. 
    \item We say that $X$ is of \tdef{complexity $\le N$} if it can be written as a union of at most $N$ sets which are $N$-basic.
    \item We say that $X$ is of \tdef{complexity $N$} if it is of complexity $\le N$ but not of complexity $\le N-1$.
\end{itemize}
\end{defn}

We denote the complexity of $X$ by $\Cpl{X}$. Generally, we regard semi-algebraic maps, or more generally, correspondences, as semi-algebraic sets via their graphs. In particular, we define the complexity of a semi-algebraic morphism $X\to Y$ (or, more generally, of a semi-algebraic correspondence) between semi-algebraic sets $X$ and $Y$, to be the complexity of the graph of the morphism (or correspondence).
%\begin{rmk}
%The precise definition of complexity here is not essential. 
%We will see below a list of properties satisfied by it, and essentially all we need is a function on semi-algebraic set that has this list of properties. 
%\end{rmk}

\subsubsection{Tame Functions} \label{ssec:Tame_Functions}
One advantage of complexity is that it can not behave wildly under reasonable operations on semi-algebraic sets. We shall formalize this philosophy via the notion of \emph{tame functions}. Let $\Salg$ be the collection of semi-algebraic sets.  
\begin{defn}{\label{def:tame}}
 Let $S\subseteq \Salg^k$ be a set of $k$-tuples of semi-algebraic sets. 
 \begin{enumerate}
     \item A function $\varphi\colon S \to \RR_{\ge 0}$ is called \tdef{tame} if for every tuple of natural numbers $N_1,\dots,N_k$ there is $N>0$ such that for every $(X_1,\dots,X_k) \in S^k$ with $\Cpl{X_i}\le N_i$ we have $\varphi(X_1,...,X_k)\le N$. 
     \item We say that a function $\Psi \colon S \to \Salg$ is tame if the composition 
     \[
     S \oto{\Psi} \Salg \oto{\Cpl{-}} \RR_{\ge 0} 
     \]
     is tame. 
 \end{enumerate}
     
\end{defn}
We shall need several concrete examples of tame functions, that we shall now present.
\begin{prop}\label{controlled_functions_generic} We have the following tame functions:
\begin{enumerate}
    \item The functions $(X,Y)\mapsto X\cap Y$ and $(X,Y)\mapsto X\cup Y$, defined on pairs of semi-algebraic sets in the same Euclidean space, are tame.
    \item The function $X\mapsto \RR^n \setminus X$ defined on arbitrary semialgebraic sets is tame.
    \item Let $S$ be the collection of all tuples $(X,Y,Z,f)$  consisting of a semi-algebraic correspondence $f\colon X\to Y$ between semi-algebraic sets and a semi-algebraic subset $Z\subseteq X$. Then, the function $S\to \Salg$ given by 
    \[
    (X,Y,f,Z)\mapsto f(Z)
    \]
    is tame.
\end{enumerate}
\end{prop}
Note that $(3)$ shows, in particular, that the formation of image and preimage under a semi-algebraic morphism is tame.
\begin{proof}
%\scomment{Todo, use the effectivity of quantifier elimination if \cite[Algorithm 14.5]{basu2algorithms}}
For $(1)$, note that if $X=C_1\cup \dots \cup C_\ell$ and $Y=C_1' \cup \dots \cup C_k'$ where the $C_i$-s and $C_j'$-s are $N$-basic, then 
\[
X\cup Y = C_1 \cup \dots \cup C_\ell \cup C_1' \cup \dots \cup C_k'
\]
is a union of at most $2N$ sets, each of which is $N$-basic, and hence it is of complexity $\le 2N$. Similarly, since $C_i \cap C_j'$ is $2N$-basic, the set 
\[
X\cap Y = \bigcup_{i,j} C_i \cap C_j'
\]
is of complexity at most $\max\{2N,N^2\}$. 

For $(2)$, assume that $X$ is of complexity $N$, and write as usual $X=C_1\cup \dots \cup C_k$ where the $C_i$-s are $N$-basic and $k\le N$. Then $\RR^n \setminus X = \bigcap_{i=1}^k \RR^n \setminus C_i$, and hence by $(1)$ it would suffice to show the claim when $X$ is $N$-basic, say 
\[
X= \{x\in \RR^n : g_1(x) = 0,\dots,g_m(x)=0,g_{m+1}(x)>0,\dots,g_N(x)>0\}.
\]
Then 
\[
\RR^n\setminus X = \cup_{i=1}^m \{x\in \RR^n : g_i(x) \ne 0\} \cup \bigcup_{i= m+1}^{N} \{x \in \RR^n : g_i(x) \le 0\}. 
\]
Hence, by $(1)$, it would suffice to bound the complexity of each of the summands in the above by a function of $N$. 
But 
\[
\{x\in \RR^n : g_i(x) \ne 0\} = \{x\in \RR^n : g_i(x) > 0\} \cup \{x\in \RR^n : g_i(x) < 0\}
\]
is of complexity $\le 2N$ and 
\[
\{x \in \RR^n : g_i(x) \le 0\} = \{x \in \RR^n : g_i(x) = 0\} \cup \{x \in \RR^n : g_i(x) < 0\} 
\]
is of complexity $\le 2N$ as well.

Finally, for $(3)$, assume that $X$,$Y$, $f$ and $Z$ are all of complexity $\le N$. We may assume without loss of generality that $X,Y\subseteq \RR^N$. Let $\Gamma_f$ be the graph of $f$. Then, 
\[
f(Z) = \pi_2(\pi_1^{-1}(Z) \cap \Gamma_f) 
\]
where $\pi_i\colon \RR^{2n}\simeq \RR^n\times \RR^n \to \RR^n$ denote the projection on the $i$-th coordinate for $i=1,2$. 
Hence, by $(1)$, it suffices to show the formation of image and preimage of semialgebraic sets along the projection $\pi\colon \RR^{2N} \to \RR^N$ are tame.
Note that if $C$ is a semi-algebraic set in $\RR^N$ given by a set of polynomial equations and inequalities 
\[
C = \{g_1(x)=0,\dots,g_\ell(x)=0,g_{\ell+1}(x)>0,\dots,g_N(x) >0\}
\] 
then
\[
\pi^{-1}(C) = \{g_1(f(x))=0,\dots,g_\ell(f(x))=0,g_{\ell+1}(f(x))>0,\dots,g_N(f(x)) >0\}
\] 
Since, for $\pi$ a linear projection, the polynomial $g_i(\pi(x))=0$ is of the same degree as $g_i$, we see that the complexity of $\pi^{-1}(Z)$ is at most twice the complexity of $Z$. 

Finally, the tameness of the function $Z\mapsto \pi(Z)$ in this case is a consequence of the efectivity of quantifier elimination in semi-algebraic geometry, see \cite[Algorithm 14.5]{basu2algorithms}. 
\end{proof}

Informally, point $(3)$ allows us to show that every function of semi-algebraic sets defined by a first order formula is tame. For our application, we only need the following special case.

\begin{lem}\label{tameness_formula}
Let $S$ be the collection of all triples $(X,Y,T)$ where $X,Y$ are semi-algebraic sets and $T\subseteq X\times Y$ is a semi-algebraic subset. The functions $S\to \Salg$ given by 
\[
(X,Y,T)\mapsto \{x\in X : \exists y\in Y, (x,y)\in T\}
\]
and
\[
(X,Y,T)\mapsto \{x\in X : \forall y\in Y, (x,y)\in T\}
\]
are tame.
\end{lem}
\begin{proof}
The first set is the image of $T$ under the projection $T\to X$, and the second is the complement of the projection of $X\times Y \setminus T$. Hence, the result follows from
\Cref{controlled_functions_generic}.
\end{proof}
As an immediate application, we can show the tameness of the formation of tangent bundle.
\begin{prop}\label{tangent_tame}
The function $X\mapsto TX$, defined on the collection of embedded affine Nash manifolds, is tame. Here, if $X\subseteq\RR^n$ we regard the tangent bundle $TX$ as a semi-algebraic subset of $\RR^{2n}$. 
\end{prop}

\begin{proof}
The tangent space of $X$ can be defined by the following formula:
\[
TX = \{(x,v)\in X\times \RR^n:  \forall \varepsilon > 0 \quad \exists \delta >0 \quad \forall 0<t<\varepsilon \quad \exists z\in X \text{ such that } |x+tv - z| < t\varepsilon\}.
\]
Applying \Cref{tameness_formula} four times to eliminate the $\forall$ and $\exists$ quatifiers, we get the result.
\end{proof}

%\begin{rmk}
%The above proof is a special case of a general procedure that shows the tameness of formulas involving the quantifiers $\forall$ and $\exists$ in valiables from a collection of semi-algebraic sets (cosidered as the inputs of a $\Salg$-valued function). We may always implement the $\exists$ quantifiers using image under a projection and the $\forall$-quantifiers using anti-image under a projection.  
%\end{rmk}

For an embedding $X\subseteq Y$ of Nash manifolds, we let $N(X;Y)$ be the normal bundle of $X$ in $Y$. This is a Nash vector bundle on $X$.  
\begin{corl}\label{normal_tame}
The function that takes a Nash submanifold $X\subseteq \RR^n$ to the normal bundle $N(X;\RR^n)\subseteq \RR^{2n}$ is tame.   
\end{corl}

\begin{proof}
The normal bundle is the image of the tangent bundle $TX$ under the correspondence 
\[
\RR^{2n} \stackrel{(x,u,v)\mapsto (x,u)}{\longleftarrow} \{(x,v,u) \in \RR^{3n}: <v,u>=0\} \stackrel{(x,u,v) \mapsto (x,v)}{\longrightarrow} \RR^{2n},
\]  
Hence, the result follows from the tameness of image under a correspondence (\Cref{controlled_functions_generic}(3)) and the tameness of the formation of tangent bundles (\Cref{tangent_tame}).
\end{proof}

\begin{corl}\label{differential_tame}
The function that takes a Nash morphism $f\colon X\to Y$ of embedded Nash manfolds to the differential $df\colon TX\to TY$, is tame. 
\end{corl}

\begin{proof}
The graph of $df$ is the tangent space of the graph of $f$. Hence, the result follows from \Cref{tangent_tame}. 
\end{proof}

\subsubsection{Avoidance for Sets with Bounded Complexity}
Our main usage of the notion of complexity is the fact that semi-algebraic sets of small compexity can always be avoided by a member of a large enough finite set. We now state this precisely and prove the relevant avoidance result. First, we show that a semialgebraic set is contained in a hypersurface of controlled degree.

\begin{prop}\label{degree_pol_vanishing_controlled}
For a semi-algebraic set $X\subseteq \RR^n$ of dimension $\le n-1$, let $d(X)$ be the minimal degree of a polynomial $g(x_1,...,x_n)$ with real coefficients, such that $g|_X = 0$. Then, the function $X\mapsto d(X)$ is tame.
\end{prop}

\begin{proof}
Let $X$ be a semi-algebraic set in $\RR^n$ of dimension $\le n-1$ and complexity $\le N$. By the definition of complexity, we may write $X=\bigcup_{i=1}^\ell C_i$
where $C_i$ is given by polynomial equations and inequalities of degree at most $N$.
First, observe that if $C_i$ is non-empty, then it can not be expressed using strict polynomial inequalities only. Indeed, otherwise $C_i$ must be open, and this cotradicts the assumption that $\dim(X)<n$. 
Hence, for each $i$ there is a polynomial $g_i$ of degree at most $N$ such that $(g_i)|_{C_i} = 0$. But then $g := \prod_i g_i$ is a polynomial of degree $\le N^N$ such that $g|_X = 0$. This shows that $d(X)\le N^N$.
\end{proof}
We now turn to our avoidance result.

\begin{blueprop}\label{avoid_small_complexity}
For every $n,N\in \NN$ there is a finite subset $S\subseteq \RR^n$ satisfying the following property: 
For every $X\subseteq \RR^n$ of complexity $\le N$ and dimension less than $n$, we have $S\nsubseteq X$.
\end{blueprop}

\begin{proof}
By \Cref{degree_pol_vanishing_controlled}, there is $d>0$, independent of $X$, and a polynomial $g\ne 0$ of degree at most $d$ such that $g|_X = 0$. Hence, it would suffice to show that there is a finite set $S$ which is not entirely contained in the zero set of any nonzero polynomial function of degree $\le d$ on $\RR^n$. In fact, every set of size ${n+1+d \choose d}$ in general position, that is, such that evaluations at the points of $S$ define independent functionals of the linear space of polynomials of degree at most $d$ in $n$ variables, satisfies this condition.
\end{proof}

\subsection{Focal Points}\label{sec:focal.points}

Recall that, for a smooth map $f\colon M_0 \to M_1$ of smooth manifolds, a \tdef{singular value} of $f$ is a point $x\in M_1$ such that there is $y\in f^{-1}(x)$ for which the differential $d_y f\colon T_yM_0 \to T_x M_1$ is not surjective.

\begin{prop} \label{singular_values_tame}
The function that takes a triple $(X,Y,f\colon X\to Y)$ of embedded Nash manifolds and a Nash morphism between them to the set of singular values of $f$ is tame. 
\end{prop}

\begin{proof}
The set of singular values of $f$ is given by the formula 

\[
\{y\in Y : \exists (x,v) \in X \times T_y Y | f(x)=y \text{ and } d_xf(v) \in TY\setminus df(TX)\}. 
\]
Hence, the result follows from the tameness of the formation of differential, tangent bundle and complement, as well as the tameness of the set of solutions to such a formula (see \Cref{tameness_formula}, \Cref{tangent_tame},  \Cref{differential_tame} and \Cref{controlled_functions_generic}(2)).
\end{proof}

For a smooth submanifold $M\subseteq \RR^n$,
we have a canonical map 
\[
\nu_M \colon N(M;\RR^n) \to \RR^n, 
\]
\[
(x,v)\mapsto x+v. 
\]
\begin{defn}\label{def:foc}
Let $M\subseteq \RR^n$ be a (not necessarily closed) smooth submanifold of $\RR^n$. We say that a point $p\in \RR^n$ is a \tdef{focal point} of $M$ if it is a singular value of $\nu_M$.
We denote by $\mdef{\Foc(M)}\subseteq \RR^n$ the set of focal points of $M$.
\end{defn}

Geometrically, a focal point of $M$ is a center of a sphere which is tangent to $M$ at a point ``twice''.
We start our analysis of focal points with the following dimension estimate.
\begin{prop} \label{dim_bound_foc_bin}
Let $X\subseteq \RR^n$ be a Nash submanifold of $\RR^n$ of dimension $\le n-1$. Then $\Foc(X)$ is a semi-algebraic subset of $\RR^n$ of dimension at most $n-1$. 
\end{prop}

\begin{proof}
The fact that $\Foc(X)$ is a semi-algebraic subset of $\RR^n$ follow from the Seidenberg-Tarski Theorem. We turn to estimate its dimension. The map $\nu_X\colon N(X;\RR^n)\to \RR^n$ is a Nash morphism of Nash manifolds, so by the semi-algebraic version of Sard's Lemma (see \cite[Theorem 4.8]{coste2000introduction}), its set of singular values $\Foc(X)$ is of dimension $\le n-1$.
\end{proof}

Our next task is to bound the complexity of the semi-algebraic set $\Foc(X)$ in terms of the complexity of $X$.
%For our application of this theory, we need the following consequence of having a bounded complexity.

\begin{blueprop} \label{focal_points_tame}
The function that takes a embedded Nash manifold $X\subseteq \RR^n$ of dimension $\le n-1$ to the set $\Foc(X)$ is tame. 
\end{blueprop}

\begin{proof}
First, by \Cref{normal_tame}, the function that takes $X$ to $N(X;\RR^n)$ is tame. Also, the formation of the map $\nu_M\colon N(X;\RR^n) \to \RR^n$ is tame, as it is the restriction of a fixed semi-algebraic morphism to $N(X;\RR^n)$. Finally, the formation of the set $\Foc(X)$ of singular values of $\nu_M$ is tame by \Cref{singular_values_tame}.
\end{proof}

\begin{corl} \label{bin_foc_avoidance}
For every $n\in \NN$ and every
$N>0$ there is a finite set $S\subseteq \RR^n$ such that, for every Nash submanifold $X\subseteq \RR^n$ of dimension $\le n-1$ and complexity $\le N$, there is $s\in S$ which in not a focal point of $X$.  
\end{corl}

\begin{proof}
By \Cref{focal_points_tame}, there is $N'>0$ such that the set of focal points of $X$ is of complexity $\le N'$. The existence of $S$ now follows from \Cref{avoid_small_complexity}. 
\end{proof}

\section{Equivalence Relations} 
In this section we discuss the general theory of equivalence relations on Nash manifolds and their quotients. Our goal is to give a practical criterion for the existence of a quotient of a Nash manifold by a Nash equivalence relation (\Cref{prop:slis.reg}).

\subsection{Abstract Equivalence Relations}
We start by discussing equivalence relation in a general categorical context.  
\begin{defn}
Let $\cC$ be a category. A pair of morphisms $R\rightrightarrows X$ exhibits $R$ as an \tdef{equivalence relation} on $X$ if, for every $Y\in \cC$, the corresponding map of sets $\Map(Y,R)\to \Map(Y,X)\times \Map(Y,X)$ is an injection and defines an equivalence relation on $\Map(Y,X)$. 
\end{defn}
If $\cC$ admits finite products, the pair of maps can be encapsulated into a single monomorphism $i\colon R\into X\times X$.
%Alternatively, a monomorphism $i\colon R\to X\times X$ is an equivalence relation if and only if it is:
%\begin{enumerate}
%    \item[\textbf{Reflexive}] The diagonal $ X\oto{\Delta} X\times X$ can be factored as 
%    \[X \to R \oto{i} X\times X.\] 
%    \item[\textbf{Symmetric}] The composition $R\to X\times X \oto{\text{swap}}X\times X$ can be factored as 
%    \[
%    R \to R \oto{i} X\times X.
%    \]
%    \item[\textbf{Transitive}] The composition $R\times_X R \to X\times X \times X \oto{\pi_2} X\times X$ can be factors as 
%    \[
%    R\times_X R \to R \oto{i} X\times X.
%    \] 
%    Here, the fibered product $R\times_X R$ is with respect to the second projection to $X$ on the first $R$ factor  and the first projection to $X$ on the second factor. The morphism $\pi_2$ is the projection on the first and third factors in the triple product.  
%\end{enumerate}
In good situations, one can form the quotient of $X$ by the equivalence relation $R$. 

\begin{defn}
Let $\cC$ be a category, and let $ R\rightrightarrows X$ be an equivalence relation. We define the \tdef{quotient} of $X$ by $R$ to be the coequalizer of the corresponding diagram
\[
\begin{tikzcd}
 {R} & {X,}
	\arrow[shift right=1, from=1-1, to=1-2]
	\arrow[shift left=1, from=1-1, to=1-2]
\end{tikzcd}
\]
when this coequalizer exists.
We denote in this case the quotient by $\mdef{X/R}$.
\end{defn}
 
 Note that, in general, the quotient $X/R$ may not exist if $\cC$ does not have all coequalizers. We will show that, under certain assumptions on the relation, the quotient of a Nash manifold by a Nash equivalence relation exists in $\Nash$.

\subsection{Restricted $\RR$-Spaces and Their Quotients}
Nash manifolds fully faithfuly embed in a larger collection of finitary geometric objects called restricted $\RR$-spaces.
\begin{defn}
A \tdef{restricted $\RR$-space} is a restricted topological space $X$ endowed with a sheaf of $\RR$-valued functions $\sO_X$.  
\end{defn}
Every Nash manifold, with its sheaf of smooth semialgebraic functions, is a restricted $\RR$-space. We denote the category of Nash manifolds by $\Nash$. 
While the quotient of a Nash manifold by an equivalence relation might not exist in $\Nash$, it does exist as a restricted $\RR$-space. More generally, we have the following:
\begin{prop}
The category of restricted $\RR$-spaces has all small colimits. In particular, every equivalence relation on a restricted $\RR$-space admits a quotient.
\end{prop}

\begin{proof} 
Let $\{(X_i,\sO_{X_i})\}_{i\in I}$ be a diagram of restricted $\RR$-spaces. We define the restricted topological space $X$ as the set-theoretic colimit of the sets $\{X_i\}_{i\in I}$. Let $\varphi_i\colon X_i \to X$ be the structure morphism to the colimit. We endow $X$ with the restricted topology for which the open sets are the subsets $U\subseteq X$ for which $\varphi_i^{-1}(U)$ is open for every $i\in I$. 
We equipt $X$ with the sheaf of functions $\sO_X$ such that a function $f\colon U\to \RR$ belongs to $\sO_X(U)$ if and only if, for every $i\in I$, the composition $f\circ \varphi_i|_{\varphi_i^{-1}(U)}$ lies in $\sO_{X_i}(\varphi_i^{-1}(U))$. 
It is easy to see that $(X,\sO_X)$ is a colimit for the diagram $(X_i,\sO_{X_i})$.
\end{proof} 

Thus, given a Nash manifold $X$ with an equivalence relation $R$, the question whether $X/R$ exists in $\Nash$ translates to the question whether the restricted $\RR$-space $X/R$ belongs to the full subcategory $\Nash$.

%We only care about the following example of a colimit. 
%\begin{example}
%Let $X$ be a restricted $\RR$-space and let $R\subseteq X\times X$ be a (set-theoretic) equivalence relation. Endow $R$ with the induced restricted $R$-space structure. Then, the quotient $X/R$, that is, the co-equalizer of 
%\[
%R\rightrightarrows X,
%\]
%is the restricted topological space $X/R$ with the sheaf of functions given by 
%\[
%\sO_{X/R}(U) = \{f\in \sO_X(\pi^{-1}(U)) : f \text{ is constant on equivalence classes}.\}
%\]
%Here, $\pi\colon X \to X/R$ is the quotient map. 
%\end{example}

\subsection{Equivalence Relations on Nash Manifolds}
Given a Nash manifold $X$ endowed with an equivalence relation $R$, we would like to show that, under suitable conditions, $X/R$ is a Nash manifold and the quotient map $q\colon X \to X/R$ is a Nash submersion. Assuming this, the implicit function theorem guarantees the existence of local sections to $q$. This observation motivates the notion of \emph{slice} discussed below. Namely, when the quotient do exists, a slice is the image of such a local section. Eventually, we use this notion in the proof of our main theorem.   
\subsubsection{Slices to Equivalence Relations on Smooth Manifolds}\label{sssec:sm} 
We start by discussing equivalence relations and their slices from a purely differential geometric perspective.      
We will impose certain conditions on our equivalence relations.
 
\begin{defn}
$R$ is \tdef{embedded}, and indicate this by writing $R\subseteq X\times X$ rather than specifying the monomorphism $i$. We further say that $R$ is \tdef{closed} if $R$ is a closed submanifold of $X\times X$.
\end{defn}

\begin{defn}
An equivalence relation $R\subseteq X\times X$ on a smooth manifold $X$ is called \tdef{submersive} if the composition  $R\subseteq X\times X \to  X$ is a submersion, where the second map is the projection to the first factor.
\end{defn}

For $x\in X$, let $R_x$ be the equivalence class of $x$ in $X$.
If $R$ is submersive then $R_x$ is a smooth submanifold of $X$ for every $x\in X$. 
A slice for $R$ is then a subset of $X$ which is transversal to all the submanifolds $R_x$ in the following sense:

\begin{defn}
Let $X$ be a smooth manifold and let $R\subseteq X\times X$ be a submersive equivalence relation on $X$. A locally closed submanifold $Z\subseteq X$ is called a \tdef{slice} for $R$, if for every $x\in Z$ we have $Z\cap R_x = \{x\}$ and $T_x X = T_xR_x \oplus T_x Z$. 
\end{defn}

An atlas of $X/R$ is then composed of slices of $R$ which cover the quotient. The latter property can be detected on $X$ itself using the saturation of the slices. 
\begin{defn}
Let $X$ be a set endowed with an equivalence relation $R$. For $Z\subseteq X$ we denote 
\[
\Sat(Z) = \bigcup_{x\in Z} R_x,
\] 
and refer to $\Sat(Z)$ as the \tdef{saturation} of $Z$ with respect to $R$. 
\end{defn}

\begin{prop} \label{slice_iso_smooth_functions}
Let $X$ be a smooth manifold and let $R$ be a submersive equivalence relation on $X$. Let $Z\subseteq X$ be a slice for $R$. Then 
\begin{enumerate}[$(1)$]
    \item The subset $U:=\Sat(Z)$ is open in $X$. 
    \item Let $f\colon U\to \RR$ be a function which is constant on equivalence classes of $R$. Then $f$ is smooth if and only if $f|_Z$ is smooth.
\end{enumerate}
\end{prop}

\begin{proof}
By \cite[\S5.9.5]{bourbaki2007varietes} there is a unique smooth manifold structure on the topological space  $X/R$ for which the quotient map $\pi:X\to X/R$ is a submersion. Since $Z$ is a slice to $R$, the map $\pi|_Z\colon Z\to X/R$ is injective. Also, since for any $x\in X$ the map $\varphi$ is constant on $R_x$, we obtain that $d_x\varphi(T_x R_x)=0$. This together with the assumption that $T_xZ \oplus T_x R_x = T_x X$, implies that 
\[
d_x(\varphi|_Z)\colon T_x Z \to T_{\pi(x)} X/R
\] 
is an isomorphism. We see that $\varphi|_Z$ is \'{e}tale and injective, and hence an open embedding of smooth manifolds. 

Now, for (1), we have $\Sat(Z) = \pi^{-1}(\pi(Z))$ which is open by the continuity of $\pi$ and the fact that $\pi(Z)$ is open in $X/R$. For $(2)$, the functions on $U$ which are constant on equivalence classes are precisely the functions on $\pi(Z)$, so we get the result from the fact that $\pi$ induces a diffeomorphism $Z\iso \pi(Z)$.  
\end{proof}

\subsubsection{Slices to Equivalence Relations on Nash Manifolds}
We now specialize to the case of Nash manifolds, and show how semialgebraic slices give semialgebraic local models for the quotient space $X/R$. 
\begin{defn}
We say that a morphism of restricted $\RR$-spaces $f\colon X\to Y$ is an \tdef{open embedding} if it is injective and open as a morphism of restricted topological spaces, and if the morphism $f^*\sO_Y \to \sO_X$ is an isomorphism of sheaves of $\RR$-valued functions on $X$.
\end{defn}

\begin{prop} \label{slice_open_immersion_quotient}
Let $X$ be a Nash manifold, and let $R \subseteq X\times X$ be a submersive equivalence relation on $X$ in $\Nash$. Let $Z\subseteq X$ be a semialgebraic slice for $R$. Then, the map $\pi|_Z\colon Z\to X/R$ is an open embedding of restricted $\RR$-spaces. 
\end{prop}

\begin{proof}
Let $\pi \colon X \to X/R$ be the quotient map.
First, by the assumption that $Z\cap R_x = \{x\}$ for every $x\in Z$, the map $\pi|_Z \colon Z\to X/R$ is injective. Next, if $U\subseteq Z$ is open, then $U$ is itself a slice of $R$, so that by part $(1)$ of \Cref{slice_iso_smooth_functions} the set $\Sat(U)$ is open in $X$. Since $\pi$ is an open map, we deduce that $\pi(U)=\pi(\Sat(U))$ is open in $X/R$. 
Hence, $\pi|_{Z}$ is an open embedding of restricted topological spaces. 

It remains to show that $\pi|_Z$ identifies the sheaf of functions on $Z$ with the sheaf of functions on its (open) image in $X/R$. By the definition of $\sO_{X/R}$, this amounts to show the following: 
\begin{itemize}
    \item[(*)] For every open set $U\subseteq \pi(Z)$, a function $g\colon \pi^{-1}(U)\to \RR$ which is constant on equivalence classes is smooth and semi-algebraic, if and only if its restriction $g|_Z$ is smooth and semi-algebraic.
\end{itemize}
The equivalence of the semi-algebraic property follows from the Seidenberg-Tarski Theorem, and the equivalence of the smoothness condition is \Cref{slice_iso_smooth_functions}(2). 
\end{proof}

In practice, it is usually easier to construct subsets which are slices only locally. This leads to the following relaxation of the notion of a slice. 

\begin{defn}\label{def:slice}
Let $X$ be a smooth manifold and let $R\subseteq X\times X$ be a submersive equivalence relation on $X$. A locally closed submanifold $Z\subseteq X$ is called an \tdef{\'{e}tale slice} for $R$ if it satisfies only the second condition in the definition of a slice, namely, if for every $x\in Z$ we have $T_x Z \oplus T_x R_x = T_xX$.  
\end{defn}

We now show that in the case of a semi-algebraic equivalence relation, the difference between \'{e}tale slices and slices disappears after passing to a finite open cover.

\begin{prop}\label{etale_slices_Zariski_slices}
Let $X$ be a Nash manifold and let $R\subseteq X\times X$ be a closed submersive equivalence relation on $X$ in $\Nash$. Let $Z\subseteq X$ be a semi-algebraic \'{e}tale slice for $R$. Then there are finitely many slices $Z_i\subseteq Z$ such that 
\[
\Sat(Z_1) \cup \dots \cup \Sat(Z_\ell) = \Sat(Z). 
\]
\end{prop}

\begin{proof}
The condition that $Z$ is an \'{e}tale slice for $R$ implies that the restricted equivalence relation $R|_Z:= R\cap Z\times Z$ is a closed \`{e}tale equivalence relation on $Z$, that is, that the projection $R|_Z \to Z$ is an \'{e}tale map. Thus, the existence of the $Z_i$-s follows as in the proof of \cite[Proposition 2.57]{bakker2018minimal}.
\end{proof}

\subsubsection{Regular Equivalence Relations}
 We shall use the following terminology regarding equivalence relations for which a ``well behaved" quotient exists: 
\begin{defn}
Let $X$ be a Nash manifold, and let $R\subseteq X\times X$ be a submersive equivalence relation on $X$ in $\Nash$. We say that $R$ is \tdef{regular} if the restricted $\RR$-space $X/R$ is a Nash manifold and the quotient map $X\to X/R$ is a Nash submersion.
\end{defn}
\begin{rmk}
Note that if $R$ is a closed, regular equivalence relation on a Nash manifold $X$, then we have a Nash diffeomorphism $R\simeq X\times_{X/R} X$. Indeed, both are closed semialgebraic submanifolds on $X\times X$ with the same set of points. 
\end{rmk}

\begin{rmk}
If $R$ is a regular equivalence relation on $X$, then the quotient $X/R$ of $X$ by $R$ in restricted $\RR$-spaces is actually also the quotient in $\Nash$. Indeed, this is a formal consequence of $\Nash$ being a full subcategory of the category of restricted $\RR$-spaces.
\end{rmk}

We finally give our criterion for an submersive equivalence relation to be regular. 
\begin{blueprop}\label{prop:slis.reg}
Let $R\subseteq X\times X$ be a submersive equivalence relation on a Nash manifold $X$. Then $R$ is regular if and only if  there exists a finite collection of \'{e}tale semi-algebraic slices $Z_1,...,Z_N$ of $R$ for which 
\[
X = \bigcup_{i=1}^N \Sat(Z_i).
\]
\end{blueprop}

\begin{proof}Let $\pi \colon X\to X/R$ be the quotient map.
\begin{itemize}
    \item Assume first that $R$ is regular, so that the quotient map $\pi$ is a Nash submersion.
By \cite[Theorem 2.4.16]{aizenbud2010rham}, there is a finite cover $X/R=\bigcup_{i=1}^N U_i$ and a collection $s_i:U_i\to X$ of smooth, semialgebraic sections of $\pi$. Then, the subsets 
\[
Z_i:=s_i(U_i)\subseteq X
\] 
form a collection of \'{e}tale slices for which $\bigcup_{i=1}^N \Sat(Z_i) = X$. Indeed, the fact that $Z_i$ is an \'{e}tale slice for $R$ follows from the assumption that $s_i$ is a smooth semiaglebraic section of $\pi$. Since $\Sat(Z_i) = \pi^{-1}(U_i)$, the fact that the open sets $\Sat(Z_i)$ cover $X$ follows from the assumption that the $U_i$-s cover $X/R$. 

\item 
Conversely, assume that there are \'{e}tale slices $Z_i$ of $R$, for which 
\[
X = \bigcup_{i=1}^N \Sat(Z_i).
\]
We wish to show that $\pi$ is a Nash submersion of Nash manifolds. 

\begin{enumerate}[Step 1.] 
\item We show that the restricted $\RR$-space $X/R$ is a Nash manifold:\\
By \Cref{etale_slices_Zariski_slices}, we may assume without loss of generality that each $Z_i$ is a slice for $R$, by replacing each $Z_i$ by a finite collection of slices for $R$. 

By \Cref{slice_open_immersion_quotient}, the map $\pi|_{Z_i}\colon Z_i \to X/R$ is an open embedding. Since the sets $\Sat(Z_i)$ cover $X$, the images of the $Z_i$-s under $\pi$ cover $X/R$. We deduce that $X/R$ can be covered by the (finitely many) Nash manifolds $\pi(Z_i)$, and hence it is a Nash manifold.

\item We show that $\pi \colon X\to X/R$ is a Nash submersion: \\
This claim depends only on the smooth manifolds $X$ and $X/R$, and not on the semialgebraic structure. Hence, the result follows from the fact that every closed submersive equivalence relation on a smooth manifold is regular (see \cite[\S 5.9.5]{bourbaki2007varietes}). 
    \end{enumerate}
 
\end{itemize}

\end{proof}

\subsubsection{Slices of Normals}
We have seen in the previous section that constructing a quotient for a submersive equivalence relation on a Nash manifold amounts to giving a large enough collection of \'{e}tale slices. In this section we give a construction of such slices for equivalence relations on $\RR^n$. Namely,  
for a submersive, semialgebraic equivalence relation $R \subseteq \RR^n \times \RR^n$, we shall construct a family of \'{e}tale slices $V(R,p)$, one for each point $p\in \RR^n$. We then show that these slices supply a finite cover of $\RR^n/R$ and use this to put on it a Nash structure.
\begin{defn}\label{defn:VRp}
Let $R$ be a closed submersive equivalence relation on $\RR^n$. For $p\in \RR^n$, let $\mdef{V(R,p)}$ be the set of points $x\in \RR^n$ satisfying the following two properties: 
\begin{itemize}
    \item The vector $p-x$ is perpendicular to the linear subspace $T_xR_x \subseteq \RR^n$. Equivalently, we have $(x,x-p)\in N_x(R_x;\RR^n)$. 
    
    \item $(x,x-p)$ is a regular point of the map $\nu_{R_x}\colon N(R_x;\RR^n) \to \RR^n$. 
\end{itemize}
\end{defn}
The sets $V(R,p)$ are useful for us because of the following:

\begin{blueprop}\label{normals_slice_is_slice}
Let $R \subseteq \RR^n\times \RR^n$ be a submersive closed equivalence relation. Then, for every $p\in \RR^n$, the subset $V(R,p)\subseteq \RR^n$ is an \'{e}tale slice for $R$, which is semi-algebraic if $R$ is semi-algebraic.
\end{blueprop}

\begin{proof}
%The fact that $|V(R,p)\cap R_x|\leq 1$ is obvius. 
By the Seidenberg-Tarski theorem, if $R$ is semi-algebraic then so is $V(R,p)$. Hence, it remains to show that:
\begin{enumerate}[(1)]
    \item\label{prop:slice:1} $V(R,p)$  is a locally closed smooth submanifold of $\RR^n$.
    \item\label{prop:slice:2} $T_xV(R,p)\oplus T_xR_x = T_x\RR^n$ for every $x\in V(R,p)$.
\end{enumerate}
\begin{itemize}
\item[Proof of (\ref{prop:slice:1}):] Let $\varphi:E\to \RR^n$ be the vector bundle given by $E_{x} = N_x(R_x;\RR^n)$. We have a map $\nu\colon E \to \RR^n$ given by $\nu(x,v) = x+v$. 
Let $s\colon \RR^n \to \RR^{2n}$ be the smooth map given by $s(x) = (x,x-p)$, and set 
$
Z:= s(V(R,p))
$ 
Then, $Z$ is an open subset of $\nu^{-1}(p)\subseteq E$, consisting of only regular points of $\nu$. Indeed, at a point $(x,v)\in Z$, even the restriction of $\nu$ to $\varphi^{-1}(R_x) = N_x(R_x;\RR^n)$, which is exactly $\nu_{R_x}$, is regular at $(x,v)$. 
By the Implicit Function Theorem, we deduce that $Z$ is a smooth submanifold of $E$. Since $Z$ is the image of $V(R,p)$ under a smooth section of the projection $\RR^{2n}\to \RR^n$, we deduce that $V(R,p)$ is itself a smooth submanifold of $\RR^n$.

\item[Proof of (\ref{prop:slice:2}):] 
Let $x\in V(R,p)$, and set $\tilde{x} = (x,x-p) \in E$. Since the submersion $\varphi \colon E\to \RR^n$ restricts to a diffeomorphism $Z\iso V(R,p)$, the condition 
\[
T_{x}V(R,p) \oplus T_x R_x = T_x \RR^n
\] 
is equivalent to the condition \[
T_{\tilde{x}}Z \oplus T_{\tilde{x}} \varphi^{-1}(R_x) = T_{\tilde{x}} E.
\] 
Since $Z$ is an open neighborhood of $\tilde{x}$ in $\nu^{-1}(p)$ and $\tilde{x}$ is a regular point of $\nu$, we have 
\[
T_{\tilde{x}}Z = \Ker(d_{\tilde{x}}\nu).
\]
Hence, to show that $T_{\tilde{x}} \varphi^{-1}(R_x)$ is a direct complement of $T_{\tilde{x}}Z$, it would suffice to show that the differential $d_{\tilde{x}}\nu$ restricts to an isomorphism 
\[
T_{\tilde{x}} \varphi^{-1}(R_x) \iso T_p \RR^n.  
\] 
In other words, we have to show that the map $\nu|_{\varphi^{-1}(R_x)}\colon \varphi^{-1}(R_x) \to \RR^n$ is \'{e}tale at $\tilde{x}$. By the definitions of $E$ and $\nu$, we have 
\[
\varphi^{-1}(R_x) = N(R_x;\RR^n)
\] 
and 
\[
\nu|_{\varphi^{-1}(R_x)} = \nu_{R_x}.
\] 
Finally, $\nu_{R_x}$ is \'{e}tale at $\tilde{x}$ since the point $\tilde{x}$ is a regular point of $\nu_{R_x}$.
\end{itemize}
\end{proof}

To use the \'{e}tale slices $V(R,p)$ as charts of the quotient $\RR^n/R$, we need to show that finitely many of them cover this quotient. Namely:

\begin{blueprop}\label{enough_slices_Rn}
Let $R\subseteq \RR^n\times \RR^n$ be a closed, submersive and semialgebraic equivalence relation on $\RR^n$. Then, there is a finite set of points $p_1,...,p_N$ in $\RR^n$ such that 
\[
\RR^n = \bigcup_{i=1}^N \Sat(V(p_i,R)).
\]
\end{blueprop}

\begin{proof}
For $p,x\in \RR^n$, if $p$ is not a focal point of $R_x$ then $x$ is equivalent to a point of $V(R,p)$. Indeed, since the equivalence class $R_x\subset \RR^n$ is closed, there is a point $y\in R_x$ which is nearest to $p$ within $R_x$. Then, $y-p$ is perpendicular to $R_x$ at $y$, and $p$ is not a focal point of $R_x$. We deduce that $x\sim_R y\in V(R,p)$. 

Consequently, we have
\[
X\setminus \Sat(V(R,p)) \subseteq \{x\in \RR^n : p\in \Foc(R_x)\}.
\]
Hence, to prove the result, it would suffice to find a finite set $S\subseteq \RR^n$ such that for every $x\in \RR^n$ there is $p\in S$ for which $p\notin \Foc(R_x)$.

By \Cref{avoid_small_complexity} and \Cref{focal_points_tame}, the existence of such $S$ will follow if we show that there is $N>0$ such that  $R_x$ is of complexity $\le N$ for every $x\in \RR^n$. 
But $R_x$ is the intersection of $R$ with an affine subspace of $\RR^{2n}$, and hence, by  \Cref{controlled_functions_generic}(1), the complexity of $R_x$ can be bounded from above by a function of the complexity of $R$ . 
\end{proof}

\subsubsection{Regularity of Closed Equivalence Relations}\label{sssec:reg.rel}
We are ready to prove the main result of the paper:

\begin{thm}\label{thm:closed_submersive_then_regular}
Let $X$ be a Nash manifold and let $R\subseteq X\times X$ be a closed, submersive and semialgebraic equivalence relation on $X$. Then $R$ is regular. 
In other words, the quotient $X/R$ is a Nash manifold and the quotient map $X\to X/R$ is a Nash submersion.
\end{thm}

\begin{proof}
By \Cref{prop:slis.reg}, it would suffice to find finitely many \'{e}tale slices for $R$ whose saturations cover $X$. By \cite[Corollary I.5.11 and Remark I.5.12]{Shiota}, there is a finite cover $X=U_1 \cup \dots \cup U_\ell$ such that $U_i\simeq \RR^n$. If $Z$ is an \'{e}tale slice for $R|_{U_i}$, then it is an \'{e}tale slice for $R$. Moreover, if we cover $U_i$ by saturation of such \'{e}tale slices, then the union of all the slices constructed for the different $U_i$-s will give a suitable collection of slices for $X$. Hence, we may replace $X$ by $U_i$ and assume without loss of generality that $X\simeq\RR^n$. 

In this case, by \Cref{enough_slices_Rn} there are finitely many points $p_1,...,p_N$ such that 
\[
X = \bigcup_{i=1}^N \Sat(V(R,p_i)).
\]
By \Cref{normals_slice_is_slice}, each $V(R,p_i)$ is an \'{e}tale slice for $R$, and hence the $V(R,p_i)$-s constitute the desired collection of slices.  
\end{proof}
%\begin{prop}
%Let $R$ be a closed submersive equivalence relation on $\RR^n$. Then, there is a finite set $S = \{p_1,...,p_r\} \subseteq \RR^n$ such that $\cup_{i=1}^r\Sat(V(R,p_i))=\RR^n$.
%\end{prop}

%\begin{proof}
%Let $k$ be the maximal degree of the  Zariski closures $\overline{R_x}$, and let $C(k,n)$ be the constant of \scomment{ref}, so that $BN(R_x)$ is contained in an algerbraic subset of $\RR^n$ of degree at most $C(k,n)$. Let $p_1,...,p_r$ be a finite set of points not contained in any subvariety of degree $C(k.n)$ of $\RR^n$ (\scomment{ref to existence}). Then, for every $x\in \RR^n$ there is $1\le i\le r$ such that $p_i$ is not a binormal of $R_x$. Let $y$ be the nearest point to $p$ in $R_x$ (there is a unique such $y$ because $R_x$ is closed in $\RR^n$ and $p_i$ is not a binormal). Then, we have $y\in V(R,p_i)$ and $y\sim x$ so that $x\in \Sat(V(R,p_i))$.  
%\end{proof}

%\bibliographystyle{unsrt}
\bibliographystyle{alpha}
\bibliography{ref}

\end{document}